\documentclass{amsart}

\usepackage[utf8]{inputenc}
\usepackage{amsmath,amssymb,amsthm}
\usepackage{mathrsfs}
\usepackage{enumerate}

\theoremstyle{definition}
\newtheorem{defi}{Definicja}[section]
\theoremstyle{plain}
\newtheorem{thm}[defi]{Theorem}
\newtheorem{lem}[defi]{Lemma}
\newtheorem{cor}[defi]{Corollary}
\newtheorem{que}[defi]{Question}

\newcommand{\U}{\mathcal{U}}
\newcommand{\pow}[1]{\mathcal{P}\left(#1\right)}
\newcommand{\X}{\mathcal{X}}
\newcommand{\A}{\mathcal{A}}
\newcommand{\I}{\mathcal{I}}

\newcommand{\F}{\mathcal{F}}
\newcommand{\FF}{\mathscr{F}}
\newcommand{\UU}{\mathscr{U}}

\newcommand{\dwa}{\{0,1\}}

\newcommand{\cl}[1]{\overline{#1}}

\newcommand{\frechet}{Fr\'{e}chet}

\renewcommand{\phi}{\varphi}

\begin{document}
\title{A filter on a collection of finite sets and Eberlein compacta}
\thanks{}

\subjclass[2010]{}

\author{Tomasz Cie\'sla}
\address{Institute of Mathematics, Faculty of Mathematics, Informatics and Mechanics, University of Warsaw}
\email{t.ciesla@mimuw.edu.pl}
\date{\today}

\maketitle

\begin{abstract}
We introduce a $\sigma$-ideal on $\omega_1 \times \omega_1$ and a filter on the collection of graphs of strictly decreasing partial functions on $\omega_1$ taking values in $\omega_1$. We use them to prove that a certain space is a non-bisequential Eberlein compactum.
\end{abstract}

\section{Introduction}

For a subset $A \subset \omega_1 \times \omega_1$ of the set of pairs of countable ordinals, we will denote the vertical and horizontal sections of $A$ at $\alpha$ by
$$A_\alpha = \{\beta \colon (\alpha, \beta) \in A \}, \qquad A^\beta=\{\alpha \colon (\alpha, \beta) \in A\},$$
and we will consider the $\sigma$-ideal
$$\I = \{A \subset \omega_1 \times \omega_1 \colon \textrm{ for all but countably many } \alpha, |A_\alpha| \le \aleph_0 \textrm{ and } |A^\alpha| \le \aleph_0 \}.$$

We will prove the following

\begin{lem} \label{istnieje-funkcja}
For any $A$ in the $\sigma$-ideal $\I$ and subsets $B_1, \ldots, B_n$ of $\omega_1 \times \omega_1$ not in $\I$, there is a strictly decreasing function $f \colon S \to \omega_1$, $S \subset \omega_1$, whose graph omits $A$ and intersects each $B_i$, $i \le n$.
\end{lem}

We will denote by $[\omega_1 \times \omega_1]^{<\omega}$ the collection of all finite subsets of $\omega_1 \times \omega_1$, endowed with the pointwise topology, i.e., identifying a set with its characteristic function, we will consider  $[\omega_1 \times \omega_1]^{<\omega}$ as the subspace of the Tychonoff product $\dwa^{\omega_1 \times \omega_1}$ consisting of functions with finite support.

Then the space
$$\X = \{A \in [\omega_1 \times \omega_1]^{<\omega} \colon A \textrm{ is the graph of a stricly decreasing function}\}$$
is an Eberlein compactum (the terminology is explained in the next section).

Peter Nyikos announced in \cite{anons-nyikosa} that the compactum $\X$ is not bisequential. However, to our best knowledge, a proof of this result has not been published.

We will show that lemma \ref{istnieje-funkcja} provides readily the following refinement of the Nyikos theorem.

\begin{thm}\label{niebiciagowy}
For any $B \subset \omega_1 \times \omega_1$ not in the $\sigma$-ideal $\I$, the Eberlein compactum $\X_B = \{A \in \X \colon A \subset B\}$ is not bisequential.
\end{thm}

Let us notice that the Eberlein compactum $\X$ was also considered by Leiderman and Sokolov \cite{leiderman}, but in a different context (cf. section \ref{uniform}).

\section{Some background}

A sequence of non-empty subsets $A_0, A_1, \ldots$ of a topological space $X$ converges to a point $x \in X$ if for any neighbourhood $U$ of $x$ there exists an integer $n_0$ such that for any $n>n_0$ we have $A_n \subset U$. An ultrafilter $\U \subset \pow{X}$ is convergent to $x \in X$ if every neighbourhood $U$ of $x$ is an element of $\U$.

We say that a topological space $X$ is bisequential if for any ultrafilter $\U \subset \pow{X}$ convergent to some element $x\in X$ there exists a sequence $U_0 \supset U_1 \supset U_2 \supset \ldots$ of elements of $\U$ convergent to $x$. Equivalently, a topological space is bisequential if and only if it is a biquotient image of metrizable space. This is due to Michael \cite[Theorem 3.D.2]{michael}.

The notion of bisequentiality is closely related to \frechet\ property. We say that a topological space $X$ is \frechet, if for any subset $A \subset X$ and $x \in \cl{A}$ there exists a sequence $x_0, x_1, \ldots$ of elements of $A$ converging to $x$. It is easy to see that every bisequential space is \frechet.

Bisequential spaces enjoy many ``good'' properties: subspaces of bisequential spaces, countable products of bisequential spaces and continuous images of compact bisequential spaces are bisequential. \frechet\ spaces behave much worse --- there exist two compact \frechet\ spaces whose product is not \frechet.

A compact space $K$ is called an Eberlein compactum if it homeomorphically embeds into some Banach space $X$ with its weak topology. If we additionally assume that $X$ is a Hilbert space, we call $K$ an uniform Eberlein compactum. 

It is not hard to see that the space $\X$ described in the introduction is a closed subspace of $\dwa^{\omega_1 \times \omega_1}$. Using the following folklore theorem
\begin{thm}\label{domknieta-rodzina-zbiorow-skonczonych}
Let $X$ be a set and let $\A$ be a family consisting of some finite subsets of $X$. If $K = \{ \chi_A \colon A \in \A \}$ is a closed subspace of $\dwa^X$ then $K$ is an Eberlein compactum.
\end{thm}
we deduce that $\X$ is an Eberlein compactum.

Note that $\X$ serves as an example of non-bisequential \frechet\ space --- this is because Eberlein compacta are \frechet\ (see e.g. \cite{archangielski}).

\section{Proof of lemma \ref{istnieje-funkcja}}

Since $A \in \I$, there exists an ordinal $\alpha < \omega_1$ such that 
$$\forall \alpha < \beta < \omega_1 \ \ |A_\beta|<\aleph_1 \wedge |A^\beta|<\aleph_1.$$
For any $i \in \{1, \ldots, n\}$ $B_i \in \pow{\omega_1 \times \omega_1} \setminus \I$, therefore there exist uncountably many ordinals $\beta < \omega_1$ such that
$$|(B_i)_{\beta}| = \aleph_1 \textrm{ or } |(B_i)^\beta| = \aleph_1.$$
Without loss of generality we can assume that $|(B_i)_{\beta}| = \aleph_1$ for $1 \le i \le k$ and $|(B_i)^\beta| = \aleph_1$ for $k+1 \le i \le n$.
Therefore there exist ordinals $\beta_1, \beta_2, \ldots, \beta_k$ such that:
\begin{itemize}
\item $\alpha < \beta_1 < \beta_2 < \ldots < \beta_k$,
\item the sets $(B_i)_{\beta_i}$ are uncountable for $1 \le i \le k$.
\end{itemize}
Analogously, there exist ordinals $\gamma_{k+1}, \gamma_{k+2}, \ldots, \gamma_n$ such that
\begin{itemize}
\item $\alpha < \gamma_n < \gamma_{n-1} < \ldots < \gamma_{k+1}$,
\item the sets $(B_i)^{\gamma_i}$ are uncountable for $k+1 \le i \le n$.
\end{itemize}
Thus there exist ordinals $\gamma_1, \gamma_2, \ldots, \gamma_k$ such that:
\begin{itemize}
\item $\gamma_{k+1} < \gamma_k < \gamma_{k-1} < \ldots < \gamma_1$,
\item $\gamma_i \in (B_i)_{\beta_i}$ for each $1 \le i \le k$,
\item $\gamma_i > \sup A_{\beta_i}$ for each $1 \le i \le k$.
\end{itemize} 
Analogously, there exist ordinals $\beta_{k+1}, \beta_{k+2}, \ldots, \beta_n$ such that
\begin{itemize}
\item $\beta_k < \beta_{k+1} < \beta_{k+2} < \ldots < \beta_n$,
\item $\beta_i \in (B_i)^{\gamma_i}$ for each $k+1 \le i \le n$,
\item $\beta_i > \sup A^{\gamma_i}$ for each $k+1 \le i \le n$.
\end{itemize}

It follows that $(\beta_i, \gamma_i) \in B_i$ and $(\beta_i, \gamma_i) \notin A$ for any $1 \le i \le n$. Moreover
$$\beta_1 < \beta_2 < \ldots < \beta_n \textrm{ and } \gamma_1 > \gamma_2 > \ldots > \gamma_n.$$

Therefore the function $f \colon \{\beta_1, \ldots, \beta_n\} \to \omega_1$ given by the formula $f(\beta_i) = \gamma_i$ has the desired properties.

\section{Proof of theorem \ref{niebiciagowy}}
Fix $B \in \pow{\omega_1 \times \omega_1} \setminus \I$. It is clear that $\X_B$ is closed subset of $\X$, and therefore is an Eberlein compactum.

For any $A, B_1, \ldots, B_n \subset B$ define
$$\F_A=\left\{C \in \X_B \colon A \cap C = \emptyset\right\}$$
and 
$$\F_A(B_1, B_2, \ldots, B_n) = \left\{C \in \X_B \colon A \cap C = \emptyset \wedge \forall i \le n \ \ B_i \cap C \neq \emptyset \right\}.$$
Let $\I_B$ be the $\sigma$-ideal $\I_B = \{A \in \I \colon A \subset B\}$. Consider
$$\FF = \{\F_A \colon A \in \I_B\} \cup \left\{\F_A(B_1, B_2, \ldots, B_n) \colon A \in \I_B \wedge n \in \omega \wedge \forall i \le n \ \  B_i \in \pow{B} \setminus \I_B \right\}.$$

Note that for any $A, A' \in \I_B$ and $B_1, \ldots, B_n, B_1', \ldots, B_m' \in \pow{B} \setminus \I_B$ we have $A \cup A' \in \I_B$ and 
\begin{align*}
\F_A \cap \F_{A'} & = \F_{A \cup A'} \\
\F_A \cap \F_{A'}(B_1, B_2, \ldots, B_n) & = \F_{A \cup A'}(B_1, B_2, \ldots, B_n) \\
\F_A(B_1, B_2, \ldots, B_n) \cap \F_{A'}(B'_1, B'_2, \ldots, B'_m) &= \F_{A\cup A'}(B_1, B_2, \ldots, B_n, B'_1, B'_2, \ldots, B'_m).
\end{align*}

It follows that intersection of any two elements of $\FF$ is an element of $\FF$. By lemma \ref{istnieje-funkcja} we infer that $\FF$ consists of non-empty sets. Therefore $\FF$ has the finite intersection property. Thus there exists an ultrafilter $\UU$ on $\X_B$ extending $\FF$. We will prove that $\UU$ witnesses non-bisequentiality of $\X_B$.

Basic neighbourhoods of $\emptyset \in \X_B$ are of the form $\F_A$, where $A$ are finite subsets of $B$ and therefore are elements of $\FF$. In particular, every neighbourhood of $\emptyset$ is an element of $\UU$ which means that $\UU$ converges to $\emptyset$.

For the sake of contradiction assume that $\X_B$ is bisequential. Then there exists a decreasing sequence $\U_0, \U_1, \ldots$ of elements of $\UU$ such that for any basic neighbourhood $\F_A$ of $\emptyset$ there exists a positive integer $i$ such that $\U_i \subset \F_A$. 

Define for any $i \in \omega$
$$A_i=\left\{(\alpha,\beta) \in B \colon \U_i \subset \F_{\{(\alpha, \beta)\}} \right\}.$$
Then $\bigcup_{i\in\omega}A_i = B$. Moreover
$$\U_i \subset \bigcap_{(\alpha,\beta)\in A_i} \F_{\{(\alpha,\beta)\}} = \F_{A_i}.$$
Thus $\F_{A_i} \in \UU$.

Suppose that $A_i \notin \I_B$ for some $i \in \omega$. Then
$$\X \setminus \F_{A_i} = \F_{\emptyset}(A_i) \in \FF \subset \UU,$$
which is a contradiction as $\UU$ is an ultrafilter and $\F_{A_i} \in \UU$. Therefore $A_i \in \I_B$ for any $i \in \omega$. 

Note that
$$\bigcup_{i \in \omega} A_i \in \I_B$$
because $\I_B$ is a $\sigma$-ideal. On the other hand
$$\bigcup_{i \in \omega} A_i = \omega_1 \times \omega_1,$$
which is a contradiction. Thus $\X_B$ is non-bisequential.

\section{Uniform Eberlein compacta} \label{uniform}

We already mentioned that the space $\X$ was also studied by Leiderman and Sokolov. They proved that $\X$ is an Eberlein compactum which isn't a uniform Eberlein compactum. In this section we present another proof of this fact using non-bisequentiality of $\X$.

We will need a few other results. 

Benyamini, Rudin and Wage gave the following characterization of uniform Eberlein compacta \cite[Lemma 1.2]{brw}:

\begin{thm}\label{benyamini}
Every uniform Eberlein compactum $K$ is a continuous image of a closed subspace of $\alpha(\lambda)^\omega$ for some cardinal $\lambda$, where $\alpha(\lambda)$ is the one-point compactification of a discrete space of cardinality $\lambda$. 
\end{thm}

Moreover, one can take as $\lambda$ the weight of $K$.

The following theorem is a well-known fact proved by Nyikos.

\begin{thm}\label{nyikos-ciezar}
If $K$ is a uniform Eberlein compactum whose weight is smaller than the first measurable cardinal, then $K$ is bisequential.
\end{thm}

\begin{proof}
Let $\lambda$ be the weight of $K$. It follows from Theorem \ref{benyamini} that there exists a closed subset $F \subset \alpha(\lambda)^\omega$ and a continuous surjection $f \colon F \to K$.

By assumption, $\lambda$ is smaller than the first measurable cardinal, therefore $\alpha(\lambda)$ is bisequential (see \cite[Example 10.15]{michael}). Thus $\alpha(\lambda)^\omega$ is bisequential as well, because it is a countable product of bisequential spaces. Thus $F$ is bisequential as a subspace of bisequential space.

Since $F$ is a closed subset of a compact space $\alpha(\lambda)^\omega$, it follows that $F$ is compact. Therefore $K=f(F)$ is a continuous image of a compact bisequential space. Thus $K$ is bisequential.
\end{proof}

\begin{cor}
$\X$ is not a uniform Eberlein compactum.
\end{cor}

\begin{proof}
$\X$ has weight $\aleph_1$ and is not bisequential. It follows from the previous theorem that $\X$ is not a uniform Eberlein compactum.
\end{proof}

\section{Comments}

One can replace $\omega_1$ by an ordinal number $\alpha$ in the description of the space $\X$. If $\alpha < \omega_1$ then the resulting space is countable, which of course is bisequential. If $\alpha > \omega_1$, then we get a non-bisequential space, because $\X$ embeds homeomorfically into it. Therefore, in a sense, $\X$ is the smallest example of a non-bisequential Eberlein compactum of this type.

We get another modification by replacing $\omega_1$ by a partially ordered set $T$. As a result we get a space which we will denote by $\X_T$. If $T$ is a tree then $\X_T$ is an Eberlein compactum.

Note that if $T$ has a branch of length at least $\omega_1$, then $\X$ can be homeomorphically embedded into $\X_T$ which implies that $\X_T$ isn't bisequential. Of course trees of height greater than $\omega_1$ have such a branch. Two interesting questions arise:

\begin{que}
For which trees $T$ of height less or equal than $\omega_1$ is $\X_T$ bisequential?
\end{que}

\begin{que}
Let $T$ be an Aronszajn tree. Is $\X_T$ bisequential?
\end{que}

The author wasn't able to answer these questions.

\section*{Acknowledgments} 

I'd like to thank Witold Marciszewski for conversations on the topic and many helpful comments. Also thanks are due to Roman Pol and Piotr Zakrzewski for remarks which simplified the proofs and greatly improved the exposition of the results.

\end{document}